\documentclass[12pt]{amsart}
\usepackage[a4paper,hmargin=3cm,vmargin=3cm]{geometry}
\usepackage{amsmath}
\usepackage{amsthm}
\usepackage{amssymb}

\theoremstyle{plain}
\newtheorem{theorem}{Theorem}[section]
\newtheorem{lemma}[theorem]{Lemma}
\newtheorem{proposition}[theorem]{Proposition}
\newtheorem{definition}[theorem]{Definition}
\newtheorem{corollary}[theorem]{Corollary}

\theoremstyle{remark}

\usepackage{xcolor}

\usepackage[
backend=biber,
style=alphabetic,
sorting=ynt
]{biblatex}

\DeclareMathOperator{\Var}{Var}
\DeclareMathOperator{\E}{\mathbb{E}}

\newcommand{\EN}{\mathbb{E}_N}
\newcommand{\PN}{\mathbb{P}_N}

\newcommand{\beforesubmit}[1]{}

\begin{document}

\title{Fluctuations of the free energy of the mixed $p$-spin mean field spin glass model}

\author{Debapratim~Banerjee, David~Belius}

\maketitle

\begin{abstract}
    We prove the convergence in distribution of the fluctuations of the free energy of the mixed $p$-spin Sherrington-Kirkpatrick model with non-vanishing $2$-spin component at high enough temperature. The limit is Gaussian, and the fluctuations are seen to arise from weighted cycle counts in the complete graph on the spin indices weighted by the $2$-spin interaction matrix.
\end{abstract}

\section{Introduction}
The Sherrington-Kirkpatrick \cite{SKSolvableModelOfASpinGlass} model and its variants \cite{derrida_random-energy_1980, gross_simplest_1984,talagrand_multiple_2000}  are important models of disordered magnetic systems and paradigmatic examples in the theory of complex systems \cite{MezardParisiVirasoro-SpinGlassTheoryandBeyond,talagrand2010meanvol1, talagrand2011meanvol2,PanchenkoTheSKModel}.

An important step in the solution of the model is the computation of the \emph{free energy}. The Parisi formula \cite{parisi1980sequence,parisi_infinite_1979,GuerraBrokenReplicaSymmetryBounds,TalagrandTheParisiFormula,PanchenkoTheParisiUltrametricityConjecture} gives the typical value of the free energy in the form of a law of large numbers. This article studies the fluctuations of free energy around its typical value. We prove that the distribution of the fluctuations for mixed $p$-spin Sherrington-Kirkpatrick models without external field at high enough temperature are asymptotically Gaussian, provided a non-vanishing $2$-spin component is present. We achieve this by proving an estimate for the free energy which is sharp to subleading order, where the subleading term arises from certain cycle counts in the complete weighted graph defined by the disorder matrix of the $2$-spin component. Thus the fluctuations of the free energy can be understood as arising from the fluctuations of these cycle counts.

To formally state our results let
\[
\xi\left(x\right)=\sum_{p\ge2}\alpha_{p}x^{p},
\]
be a non-zero power series with radius of convergence greater than one with $\alpha_{p}\ge0$ for all $p$, called the \emph{mixture}. Let $H_{N}\left(\sigma\right)$
be a centered Gaussian process on the sphere $\{ \sigma \in \mathbb{R}^n: |\sigma| = \sqrt{N}\}$ with covariance
\begin{equation}\label{eq: hamilt covar}
\mathbb{E}_{N}\left[H_{N}\left(\sigma\right)H_{N}\left(\sigma'\right)\right]=N\xi\left(\frac{\sigma\cdot\sigma'}{N}\right),
\end{equation}
called a \emph{mixed $p$-spin Hamiltonian} with mixture $\xi$.
The left-hand side is a well-defined covariance  function on the sphere for any $\xi$ by Schonenberg's theorem \cite{schoenberg1942positive}, and can be explicitly constructed as a polynomial with Gaussian random coefficients, see e.g. \cite[(1.12)-(1.15)]{PanchenkoTheSKModel}.

Let $E$ be the uniform measure on $\{\pm1\}^N$ and let the free energy be given by
\begin{equation}\label{eq: free energy def}
F_N = \log Z_N \text{, where }Z_N = E\left[\exp\left(\beta H_N(\sigma)\right)\right],
\end{equation} 
for an inverse temperature $\beta \ge 0$.
Let $\mathcal{N}(\mu,\sigma^2)$ denote the normal distribution with mean $\mu$ and variance $\sigma^2$ and write $\overset{D}{\to}$ for convergence in distribution. We have the following result for the fluctuation of the free energy:
\begin{theorem}
\label{thm: main thm}For any $\xi$ with $\alpha_2 > 0$ there exists a $\beta_{\xi} \in (0,\frac{1}{\sqrt{2\alpha_2}}]$
such that if $0<\beta<\beta_{\xi}$ then letting
$$ s^2= -\frac{1}{2}\log\left(1-2\alpha_2\beta^{2}\right),$$
we have
\begin{equation}\label{eq: FE conv in law}
F_{N}-N\frac{\beta^{2}}{2}\xi(1) \overset{D}{\to}\mathcal{N}\left(-\frac{1}{2}s^2,s^2\right),
\end{equation}
as $N\to\infty$. If $\xi(x)=\alpha_2 x^2$ then $\beta_{\xi} = \frac{1}{\sqrt{2\alpha_2}}$.
\end{theorem}
The leading order term in \eqref{eq: FE conv in law} comes from the annealed partition function 
\begin{equation}\label{eq: ann Z}
\EN[Z_N] = \exp\left( N\frac{\beta^2}{2}\xi(1) \right).
\end{equation}
Theorem \ref{thm: main thm} will follow from a precise estimate of $\log Z_N$ to subleading order. To state it we construct the \emph{$2$-spin component} $H^2_N$ by setting
\begin{equation}\label{eq: HN2 def}
H_{N}^{2}\left(\sigma\right)= \frac{1}{\sqrt{2N}} \sum_{i,j}J_{ij}\sigma_{i}\sigma_{j},
\end{equation}
for $J=(J_{ij})_{i,j=1,\ldots,N}$ a GOE random matrix, i.e.
\begin{equation}\label{eq: J tilde law}
    J_{ii} \sim \mathcal{N}(0,2)\text{ for all }i, \, J_{ji} = J_{ij} \sim \mathcal{N}(0,1) \text{ for } i\ne j
    ,\, J_{ij} \text{ independent for }i\le j,
\end{equation}
defined under a probability $\PN$.
Note that then $\EN[H_N^2(\sigma)H_N^2(\sigma')] = N(\sigma \cdot \sigma'/N)^2$. We construct $H_N(\sigma)$ by letting
\[
H_{N}\left(\sigma\right)=\sqrt{\alpha_2}H_{N}^{2}\left(\sigma\right)+\tilde{H}_{N}\left(\sigma\right),
\]
where $\tilde{H}_{N}\left(\sigma\right)$
is a centered Gaussian process in $\mathbb{R}^{N}$, independent of $H_N^2$ and also defined under $\PN$, with covariance $\EN\left[\tilde{H}_{N}\left(\sigma\right)\tilde{H}_{N}\left(\sigma'\right)\right]=N\tilde{\xi}\left(\frac{\sigma\cdot\sigma'}{N}\right)$
for $\tilde{\xi}\left(x\right)=\xi\left(x\right)-\alpha_2 x^{2}$. The process $H_N(\sigma)$ then satisfies \eqref{eq: hamilt covar}.

Having constructed the $2$-spin component $H_N^2$, consider now its centered weighted cycle counts
\begin{equation}\label{eq: cycle count def}
C_{N,k}=\frac{1}{N^{\frac{k}{2}}}\sum_{i_0,\ldots,i_{k-1}\text{ distinct}}J_{i_{0}i_{1}}\ldots J_{i_{k-1}i_{0}}
- (N-1)1_{k=2},\quad k=1,2,\ldots.
\end{equation}
The sum in \eqref{eq: cycle count def} has mean zero if $k\ne2$ since each summand is a product of distinct independent centered Gaussians, and if $k=2$ it equals $\sum_{i_0 \ne i_1} J_{i_0i_1}^2$ which has mean $N(N-1)$. Note that $C_{N,k}=0$ for $k>N$.

Finally letting
\begin{equation}\label{eq: muk def}
\mu_{k}=\left( \sqrt{2\alpha_2}\beta\right)^{k},k =1,2,\ldots,
\end{equation}
and writing $\overset{P}{\to}$ for convergence in probability we have the following precise estimate for the free energy:
\begin{theorem}
\label{thm: sec thm}For any $\xi$ with $\alpha_2 > 0$ there exists a $\beta_{\xi} \in (0,\frac{1}{\sqrt{2 \alpha_2}}]$
such that if $0<\beta<\beta_{\xi}$ then
\begin{equation}\label{eq: FE estimate}
\left|F_{N}-N\frac{\beta^{2}}{2} \xi(1) - \sum_{k=1}^{\infty}\frac{2\mu_{k}C_{N,k}-\mu_{k}^{2}}{4k}\right|\overset{P}{\to}0,
\end{equation}
as $N\to\infty$. If $\xi(x)=\alpha_2 x^2$ then $\beta_{\xi} = \frac{1}{\sqrt{2\alpha_2}}$.
\end{theorem}

Theorem \ref{thm: sec thm} identifies the origin of the fluctuations of the free energy as fluctuations of the cycle counts arising from the 2-spin component.
\subsection{Previous work}
To the best of our knowledge the first work on fluctuations of the free energy was \cite{ALR}, which obtained the fluctuations in the standard Sherrington-Kirkpatrick model (that is the case $\xi(x)=x^2$) up to the critical temperature. It derived both the leading order of the free energy and the fluctuations through a graphical analysis, in which cycle counts give the main contribution. In the present work we rely on the second moment method for the leading order, and find a different way to use graphical analysis to study the fluctuations (see Section \ref{sec: outline} below for more details).

The work \cite{bovier2002fluctuations} obtained the law of the fluctuations for pure $p$-spin Hamiltonians (that is for the case $\xi(x)=x^p$ for some $p\ge2$) for small enough $\beta$ using a martingale method. Furthermore \cite{ChenDeyPanchenkoFluctFEmixedpspinextfield} did the same for mixed $p$-spin Hamiltonians without odd $p$-terms and with non-zero external field at all temperatures using a combination of interpolation, the Chen-Stein method and the Parisi formula. For the spherical SK model the fluctuations of the free energy at high temperature has been obtained in \cite{BaikLeeFluctuations} using the random matrix techniques that are applicable to that special case. Similar techniques are used in the related works \cite{Landon2SpinSphericalCriticalFluct,LandonFluct2SpinExtField,Landon2spinFluctLowTemp}.

\subsection{Outline of proof}\label{sec: outline}
We now give a high-level sketch of the proof of the estimate Theorem \ref{thm: sec thm} of $\log Z_N$ in terms of cycle counts. The method has previously been used by the first author to study fluctuations in a stochastic block model \cite{Ban16}, %
in a hypothesis testing problem for spiked random matrices
\cite{BanMa18} and in the SK model with Curie-Weiss interaction \cite{Ban18}.

It is known since \cite{ALR} that the subleading fluctuations of the pure $2$-spin model are determined by the cycle counts $C_{N,k}$. In the approach of \cite{ALR} the cycle counts appear as the leading contributions in a graphical cluster expansion. Cycle counts are also relevant in the study of fluctuations in the stochastic block model \cite{MNS12} and random regular graphs \cite{Jan}. In this approach the cycle counts enter the analysis via certain Radon-Nikyodym derivatives. Our method is inspired by the latter approach.

Once one suspects that the cycle counts determine the fluctuations of $Z_N$ one can guess the form of the fluctuations as follows: One views $\hat{Z}_N = Z_N/\mathbb{E}_N[Z_N]$ as a Radon-Nikodym derivative $\frac{d\mathbb{P}_N}{d\mathbb{Q}_N}$ and considers the law of the sequence
\begin{equation}\label{eq: cycle count sequence}
C_{N,1}, C_{N,2}, C_{N,3}\ldots,
\end{equation}
of cycle counts under the measures $\mathbb{Q}_N$.
Under the measure $\PN$ one can show using the moment method that the cycle counts are asymptotically Gaussian:
\begin{equation}\label{eq: cycle count CLT}
     C_{N,k} \overset{D}{\to} \mathcal{N}(0, 2k) \text{ under } \PN,
\end{equation}
jointly for finitely many $k$, with the $C_{N,k}$ becoming independent in the limit. For more details on this important computation see the next subsection and Proposition \ref{prop: cycle counts under P}. It turns out that $\hat{Z}_N$ as a Radon-Nikodym derivative changes the law of the $J_{ij}$ in a simple way: it gives them a random non-zero mean, but otherwise the law stays the same (Lemma \ref{lem: rad niko}, Corollary \ref{cor: QN mixture}). Using this one can show that under $\mathbb{Q}_N$ the cycle counts are also asymptotically Gaussian but with a non-zero mean, namely
\begin{equation}\label{eq: cycle count CLT Q}
     C_{N,k} \overset{D}{\to} \mathcal{N}(\mu_k, 2k) \text{ under } \mathbb{Q}_N,
\end{equation}
jointly for finitely many $k$, still independent in the limit (see Proposition \ref{prop: cycle counts under Q}).

Recall that if of $C\sim\mathcal{N}\left(0,\sigma^{2}\right)$
then the Radon-Nikodym derivative $\exp\left(\frac{2\mu C - \mu^2 }{2\sigma^{2}}\right)$
changes the law of $C$ to $\mathcal{N}\left(\mu,\sigma^{2}\right)$.
Now if under a measure $\mathbb{P}$ the sequence \eqref{eq: cycle count sequence} is exactly independent Gaussian with mean $0$ and variance of $k$-th variable given by $2k$, and $\mathbb{Q}$ is the measure where they have the same distribution but with the mean of the $k$-th variable given by $\mu_k$ instead of $0$, then we necessarily have
\begin{equation}\label{eq: rad niko intro}
\frac{d \mathbb{P}}{d \mathbb{Q}}
= \exp\left( \sum_{k=1}^{\infty}\frac{2\mu_{k}C_{N,k}-\mu_{k}^{2}}{4k} \right).
\end{equation}

Thus a possible approximation of $\log \hat{Z}_N$ is as $\log$ of the right-hand side, which is precisely the sum that appears in \eqref{eq: FE estimate}. We prove Theorem \ref{thm: sec thm} by making this approximation rigorous and using \eqref{eq: ann Z}.

To achieve this we roughly speaking use a second moment estimate for the difference between the normalized partition function $\hat{Z}_N$ and the the RHS of \eqref{eq: rad niko intro}, together with the Chebyshev inequality. Since for finite $N$ the cycle counts have tails that decay too slowly for the RHS of \eqref{eq: rad niko intro} to have a finite second moment, we carry out this argument on a limiting probability space where after taking $N\to\infty$ the cycle counts become exactly Gaussian and thus do have finite exponential moments. The second moment argument depends on the second moment $\EN[Z_N^2]$ of the partition function being asymptotic to $c_\xi \EN[Z_N]^2$, where $c_\xi$ is a constant depending on $\xi$, and therefore works precisely when the vanilla second moment method proves that the leading order free energy is given by its annealed value.

Theorem \ref{thm: main thm} is a simple consequence of Theorem \ref{thm: sec thm} and the asymptotic normality and independence of the cycle counts.

\subsection{Cycle counts} 
As mentioned above a crucial step in both the proof of Theorem \ref{thm: sec thm} and the derivation of Theorem \ref{thm: main thm} is the asymptotic normality \eqref{eq: cycle count CLT} of the cycle counts, which we prove in Proposition \ref{prop: cycle counts under P}.

The cycle counts are related the to traces of a power of a GOE random matrix studied in random matrix theory; indeed if the sum in \eqref{eq: cycle count def} is taken over all $i_0,\ldots,i_{k-1}$ without the requirement that they be distinct then this sum is precisely $N^{-k/2} \text{Tr}(J^k)$. It is well-known that the traces satisfy a CLT with the same normalization $N^{-k/2}$ but different recentering (see e.g. \cite[Chapter 1]{sinai1998central,AZ05,AGZ}). The traces do not become asymptotically independent, and the variance of $N^{-k/2}  \text{Tr}(J^k)$ is different from that of the corresponding cycle count. To prove \eqref{eq: cycle count CLT} we adapt in Section \ref{sec: mom} the random matrix method to study traces, namely the moment method together with a graphical computation of the moments. The latter computation turns out to be simpler for cycle counts than for traces, since the restriction to distinct $i_0,\ldots,i_{k-1}$ leads to a simpler collection of graphs, namely only cycles.
\subsection{Discussion}
It is natural to ask for how large $\beta$ the claims of Theorems \ref{thm: main thm}-\ref{thm: sec thm} remain true. Let $\beta_*$ denote the supremum of all such $\beta$.

One may note that the recentering and variance in \eqref{eq: FE conv in law} explodes as $\beta$ approaches $\frac{1}{\sqrt{2\alpha_2}}$, so that certainly $\beta_* \le \frac{1}{\sqrt{2\alpha_2}}$. We must also have $\beta_* \le \beta_c$, where $\beta_c$ is the critical inverse temperature of the static phase transition for the Hamiltonian, since for $\beta>\beta_c$ even the leading order term in \eqref{eq: FE conv in law} is incorrect.

As alluded to above, the $\beta_{\xi}$ in our theorems is the largest inverse temperature for which the vanilla second moment proves a lower bound for the free energy, and thus we have $\beta_\xi = \beta_* = \beta_c$ only for the pure $2$-spin model, and otherwise we expect that $\beta_\xi < \beta_* \le \beta_c$.

An interesting question is whether \eqref{eq: FE conv in law} holds all the way up to $\beta_c$, i.e. if $\beta_* = \beta_c$, as is the case for the 2-spin model, or if there is a second regime with different fluctuations at high temperature, as is the case for the REM model \cite{bovier2002fluctuations}.

\section{Asymptotic normality of cycle counts using method of moments and Wick's formula}\label{sec: mom}
Recall that $\mathbb{P}_{N}$ is the probability of the probability space on which $J$ and $\tilde{H}_N$ are defined. In this section we will prove the following result on the convergence in distribution of the weighted
cycle counts $C_{N,k}$ from \eqref{eq: cycle count def} under the measure $\mathbb{P}_{N}$.

\begin{proposition}[Limiting law of centered cycle counts under $\mathbb{P}_N$]
\label{prop: cycle counts under P}For any $k\ge1$
\begin{equation}\label{eq: cylce counds under P}
\mathbb{P}_{N}\text{-law of }\left(C_{N,1},C_{N,2},\ldots,C_{N,k}\right)\overset{D}{\to}\left(C_{\infty,1},C_{\infty,2},\ldots,C_{\infty,k}\right),
\end{equation}
where $\left(C_{\infty,1},\ldots,C_{\infty,k}\right)$ is a centered
independent Gaussian vector where $C_{\infty,k}$ has variance $2k$.
\end{proposition}

The proof will use the method of moments, Wick's formula and the combinatorial framework from \cite{AZ05} and \cite[Chapter 1]{AGZ}. We first state some elements of the latter framework.

\begin{definition}[Word]
For a given $N\ge1$, a letter is an element of $\{1,\ldots,N\}$. A word $w$ is a finite sequence of letters $s_1 \ldots s_n$, at least one letter long.
A word $w$ is closed if its first and last letters are the same.
\end{definition}

For any word $w = (w_0 \ldots w_{k-1})$, we use $l(w) = k$ to denote the length of $w$ %
and $\mathrm{supp}(w)$ to denote the support of $w$, i.e. the set of letters appearing in $w$. To any word $w$ we may associate a graph as follows.
\begin{definition}[Graph associated with a word]
Given a word $w = (w_0, \ldots, w_{k-1})$,
we let $G_w = (V_w,E_w)$ be the graph with vertex set $V_w = \mathrm{supp}(w)$ and edge set $E_w = \{\{w_i, w_{i+1}
\}; i = 0,\ldots ,k - 2
\}.$
\end{definition}
Note that $G_w$ is an undirected simple graph permitting loops.
The word $w$ defines a walk on the graph $G_w$ which further starts and terminates at the same vertex if the word is closed. For $e \in E_w$, we use $N^w_e$ to denote the number of times this walk traverses
the edge $e$ (in any direction).

In this paper we shall mainly deal with a special class of words, namely cyclic words.
\begin{definition}[Set $\mathfrak{W}_{l}$ of cyclic words]\label{def: cyclic word}
We call a word $w$ cyclic if $l(w)=2,3$ and the word is closed, or if $l(w)\ge4$ and the word is closed, the graph $G_{w}$ is a cycle and $N^w_e = 1$ for each edge $e$ in $G_{w}$. We write $\mathfrak{W}_{l}$ for the set of all such words of length $l$.
\end{definition}
Note that for $k\ge1$ we have $w \in \mathfrak{W}_{k+1}$ iff $w = (i_0,\ldots,i_{k-1},i_0)$ for $i_0,\ldots,i_{k-1}$ distinct. Thus
\begin{equation}\label{eq: size of Wk}
|\mathfrak{W}_{k+1}|=\begin{cases}
N(N-1)\ldots(N-k+1) & \text{ for }0\le k\le N,\\
0 & \text{\,for }k>N.
\end{cases}
\end{equation}
Also we see that the sum in \eqref{eq: cycle count def} is exactly a sum over all $(i_0,\ldots,i_{k-1},i_0) \in \mathfrak{W}_{k+1}$, so that

\begin{equation}\label{eq: cycle count in terms of cylic word non cent weight}
    C_{N,k} = \frac{1}{N^{\frac{k}{2}}}\sum_{w \in \mathfrak{W}_{k+1}} J_{w} - (N-1)1_{\{k=2\}}\text{ for }k\ge1,
\end{equation}
where we define the weight 
\begin{equation}\label{eq: word weight def not cent}
J_w = \prod_{i=0}^{l(w)-1} J_{w_{i}w_{i+1}},
\end{equation}
of a cyclic word $w$. We define also the centered weight
\begin{equation}\label{eq: word weight def}
\hat{J}_w = J_w -\EN[{J}_w]
=\begin{cases}
J_{w} & \text{\,if }l(w)\ne 3,\\
J_{w_{0}w_{1}}^{2}-1 & \text{\,if }l(w)=3,
\end{cases}
\end{equation}
of a cyclic word $w$ (note that the words $w\in\mathfrak{W}_{k+1}$ for $k=2$ are special since they satisfy $\EN[J_w]=\EN[J_{w_0w_1}^2]=1$, otherwise $\EN[J_w]=0$). Using this and for the case $k=2$ that $|\mathfrak{W}_3|=N(N-1)$ we obtain from \eqref{eq: cycle count in terms of cylic word non cent weight} the formula
\begin{equation}\label{eq: cycle count in terms of cyclic words}
C_{N,k} = \frac{1}{N^{\frac{k}{2}}}\sum_{w \in \mathfrak{W}_{k+1}} \hat{J}_{w},k\ge 1.
\end{equation}

To make use of this formula we will use the following properties of centered word weights.
\begin{lemma}[Properties of centered word weights]\label{lem: weight prop}
For all cyclic words $w$
\begin{equation}\label{eq: word weight mean}
    \EN[\hat{J}_w]=0.
\end{equation}
Furthermore for all cylic words $w,v$
    \begin{equation}\label{eq: word weight covar}
\EN[\hat{J}_{w}\hat{J}_{v}]
=
\begin{cases}
0 & \text{\,if }G_{w}\ne G_{v},\\
a_{k} & \text{ if }G_{w}=G_{v}\text{ and }w,v\in\mathfrak{W}_{k+1}\text{\,for }k\ge1,
\end{cases}
\end{equation}
where $a_1 = a_2=2, a_k=1,k\ge 3$.

Lastly for any sets $A,B$ of cyclic words
\begin{equation}\label{eq: word weight indep}
    \left(\hat{J}_{w}\right)_{w\in A}\text{\,is independent of }\left(\hat{J}_{v}\right)_{v\in B}\text{ if }\left(\bigcup_{w\in A}E\left(G_{w}\right)\right)\cap\left(\bigcup_{v\in B}E\left(G_{v}\right)\right)=\emptyset.
\end{equation}    
\end{lemma}
\begin{proof}
The claim \eqref{eq: word weight mean} is immediate from the definition \eqref{eq: word weight def} of $\hat{J}_w$.

Turning to \eqref{eq: word weight covar}, recall first \eqref{eq: J tilde law}. Note that if $G_w\ne G_v$ then there is an edge $e=\{i,j\}$ that is in only one of $G_w$ and $G_v$, so that $J_{ij}=J_{ji}$ appears in the product $\hat{J}_w \hat{J}_v$ exactly once, as either a factor $J_{ij}$ or a factor $J_{ij}^2-1$ for $i\ne j$, both of which have mean zero, so that the independence of the $J_{ij}$ implies that $\EN[ \hat{J}_w \hat{J}_v ]=0$. Now consider the case $G_w = G_v$, which can only occur if $w,v\in \mathcal{W}_{k+1}$ for $k\ge 1$. If $k\ge3$ then $\hat{J}_w \hat{J}_v = \prod_{ \{i,j\} \in E(G_w)} J_{ij}^2$ which has mean $1$, and if $k=2$ then $\hat{J}_w \hat{J}_v = \left(J_{w_0w_1}^2-1\right)^2$ which has mean $2$ and finally if $k=1$ then $\hat{J}_w \hat{J}_v = J_{w_0w_0}^2$ which has mean $2$. This proves \eqref{eq: word weight covar}.

The claim \eqref{eq: word weight indep} follows because if the words in $A$ do not share any edges with the words in $B$ then there is no random variable $J_{ij}$ that appears in both in $\hat{J}_w$ for some $w \in A$ and in $\hat{J}_v$ for some $v\in B$.
\end{proof}

We now compute the mean and the variance $C_{N,k}$ using \eqref{eq: cycle count in terms of cyclic words} and the previous lemma.

\begin{lemma}[Mean and variance of $C_{N,k}$]\label{lem: means and variances}
For all $N$ it holds under $\PN$ that
\begin{equation}\label{eq: means}
\EN[C_{N,k}] = 0\text{ for all }k\ge1,
\end{equation}
and
\begin{equation}\label{eq: covar}
    \EN[ C_{N,k}C_{N,l}] =0\text{ for all }k\ne l,
\end{equation}
and
\begin{equation}\label{varcnk}
\Var[C_{N,k}]= 
\begin{cases}
2k \frac{N(N-1)\ldots (N-k+1)}{N^{k}} & \text{ for } k\le N, \\
0 & \text{ for }k>N.
\end{cases}
\end{equation}
For fixed $k\ge1$
\begin{equation}\label{eq: variances}
    \Var[ C_{N,k} ]  \to 2k \mbox{ as }N\to\infty.
\end{equation}
\end{lemma}
\begin{proof}
The claim \eqref{eq: means} follows from \eqref{eq: cycle count in terms of cyclic words} and \eqref{eq: word weight mean}. To compute the (co-)variances note that for all $k,l$
\begin{equation}
    \EN[C_{N,k}C_{N,l}]= \frac{1}{N^{\frac{k+l}{2}}}\sum_{w\in \mathfrak{W}_{k+1},v\in \mathfrak{W}_{l+1}}\EN[\hat{J}_{w}\hat{J}_{v}].
\end{equation}
Recalling \eqref{eq: word weight covar} we note that since $G_w\ne G_v$ if $w \in \mathfrak{W}_{k+1},v \in \mathfrak{W}_{l+1}$ for $k\ne l$ the claim \eqref{eq: covar} follows. Setting $w=v$ we get for $k\ge 1$
\begin{equation}
    \EN[C_{N,k}^2] = \frac{a_k}{N^k} \sum_{w \in \mathfrak{W}_{k+1}}  |\{v \in \mathfrak{W}_{k+1}: G_v = G_w \}|,
\end{equation}
for $a_k$ as in \eqref{eq: word weight covar}. If $G_w = G_v$ then the sequence $v$ must be a walk of the graph $G_w$ of length $k+1$ that visits all $k$ vertices of $G_w$ and ends at the vertex where it started.

For $k=1$ there is one such walk, so $|\{v \in \mathfrak{W}_{k+1}: G_v = G_w \}|=1$ which together with $|\mathfrak{W}_2|=N$ and $a_1=2$ gives $\Var[C_{N,1}]=2$ and proves \eqref{varcnk} for $k=1$.

If $k=2$ there is one such walk for each of the two possible starting vertices, so $|\{v \in \mathfrak{W}_{k+1}: G_v = G_w \}|=2$ which together with $|\mathfrak{W}_3|=N(N-1)$ and $a_k=2$ gives $\Var[C_{N,2}]=4(N-1)/N$ and proves \eqref{varcnk} for $k=2$.

If $k\ge3$ then all such walks can be enumerated by picking one of $k$ starting vertices, and then picking one of two directions to traverse the cycle. Therefore  $|\{v \in \mathfrak{W}_{k+1}: G_v = G_w \}|=2k$, so that with $a_k=1$ we get
\begin{equation}
\Var[C_{N,k}]= \frac{2k|\mathfrak{W}_{k+1}|}{N^{k}},
\end{equation}
which implies \eqref{varcnk} for $k\ge 3 $ by \eqref{eq: size of Wk}. Finally \eqref{eq: variances} is a simple consequence of \eqref{varcnk}.
\end{proof}

This shows that the mean and covariance of the vector on the LHS of \eqref{eq: cylce counds under P} converge to the those of the vector on the RHS. To prove the convergence in law we will verify the convergence of higher moments and use the method of moments in the form we now state.

\begin{lemma}[Method of moments]\label{lem:mom}
Let $(Y_{N,1},\ldots, Y_{N,l}), N\ge 1,$ be a sequence of random vectors of dimension $l$. 
Assume that:
\begin{enumerate}
\item[1)] (Mixed moments converge) For any fixed $m$ and $i_1,\ldots,i_m \in \{1,\ldots, l\}$ the limit
\begin{equation}\label{eqn:momcond}
\lim_{N \to \infty}\E[Y_{N,i_1}\ldots Y_{N,i_m}] 
\end{equation}
exists.
\item[2)] (Carleman's Condition; \cite{Carl26})
It holds
$\sum_{h=1}^{\infty} \left(\lim_{N \to \infty}\E[Y_{N,i}^{2h}]\right)^{-\frac{1}{2h}} =\infty$ for all $1\le i \le l$.
\end{enumerate} 
Then the vector $(Y_{N,1},\ldots, Y_{N,l})$ converges in distribution to some random vector $(Z_1,\ldots,Z_l)$. Further, if $m,i_1,\ldots,i_m$ are as in 1) then the limit in \eqref{eqn:momcond} equals $\E[Z_{i_1}\ldots Z_{i_m}]$.
\end{lemma} 
Next we recall Wick's formula.
\begin{lemma}[Wick's formula; \cite{W50}]\label{lem:wick}
Let $(Y_1,\ldots, Y_{l})$ be a centered random vector of dimension $l$ with covariance matrix $\Sigma$ (possibly singular). Then $(Y_1,\ldots, Y_{l})$ is jointly Gaussian if and only if for any integer $m$ and and $i_1,\ldots,i_m \in \{1,\ldots,l\}$ we have with $X_{r} = Y_{i_r}$ that
\begin{equation}\label{eqn:wick}
\E[X_1\ldots X_{m}]=
\begin{cases}
      \sum_{\eta} \prod_{\{i,j\} \in \eta } \E[X_i X_j]& ~ \text{for $m$ even}\\     
      0 & \text{for $m$ odd,}
\end{cases}
\end{equation}
where the sum is over pairings $\eta$ of $\{1,\ldots,m \}$ (that is of partitions of this set into $\frac{m}{2}$ sets containing exactly $2$ elements). \end{lemma} 

To compute the higher and mixed moments of the $C_{N,k}$ we will use further combinatorial concepts from \cite{AZ05}.

\begin{definition}[Sentences and corresponding graphs]
A sentence $a=[w_i]_{i=1}^{n}=[[\alpha_{i,j}]_{j=1}^{l(w_i)}]_{i=1}^{n}$ is an ordered collection of $n$ words $w_1,\ldots,w_n$ of length $(l(w_1),\ldots,l(w_n))$. We define the graph $G_a=(V_a,E_a)$ to be the graph with 
\[
V_a= \mathrm{supp}(a), E_a= \left\{ \{ \alpha_{i,j},\alpha_{i,j+1}\}| i=1,\ldots,n ; j=1,\ldots, l(w_i)-1 \}  \right\}. 
\] 
We define an equivalence relation on sentences by saying that sentences $a$ and $b$ are equivalent if there is a permutation of $\{1,\ldots,N\}$ which turns $a$ into $b$ when applied to each letter of each word of $a$.
\end{definition} 
From \eqref{eq: cycle count in terms of cyclic words} one sees that a mixed moment of $C_{N,l}$-s can be written as a sum over sentences:
\begin{equation}\label{eq: mixed mom in terms of sentence}
\EN[ C_{N,l_1}\ldots C_{N,l_m} ]
= \sum_{a \in \mathfrak{W}_{l_1+1}\times \ldots \times \mathfrak{W}_{l_m+1} } \EN\left[ \hat{J}_a \right],
\end{equation}
where
\begin{equation}\label{eq: sentence weight def}
\hat{J}_a = \prod_{i=1}^m \hat{J}_{a_i},
\end{equation}
is the weight of a sentence $a$ of length $m$.

\begin{definition}[Weak CLT sentences]
A sentence $a=[w_i]_{i=1}^{n}$ is called a weak CLT sentence if the following conditions are true:
\begin{enumerate}
\item[1)] All the words $w_i$ are closed.
\item[2)] Jointly the words $w_i$ visit each edge of $G_a$ at least twice.
\item[3)] For each $i\in \{1,\ldots,n \}$, there is another $j\neq i \in \{ 1,\ldots,n\}$ such that $G_{w_i}$ and $G_{w_j}$ have at least one edge in common.
\end{enumerate}
\end{definition}

We have that
\begin{equation}\label{eq: zero if not weak clt}
    \EN\left[ \hat{J}_a \right]=0 \text{ if } a \text{ is a sentence of cylic words and not a weak CLT sentence},
\end{equation}
since a sentence of cyclic words that is not a weak CLT sentence must violate either 2) or 3); if 2) is violated then the product $\hat{J}_a$ contains some $J_{ij},i\le j$ exactly once, so that $\EN[J_a]=0$ by the independence of the $J_{ij}$, and if 3) is violated then there is an $i$ such that $G_{a_i}$ is disjoint from $\cup_{j\ne i} G_{a_j}$ so that $\EN[\hat{J}_a] = \EN\left[\prod_{j\ne i}\hat{J}_{a_j}\right] \EN[\hat{J}_{a_i}] =0$ by \eqref{eq: word weight indep} and \eqref{eq: word weight mean}.

By \eqref{eq: zero if not weak clt} only weak CLT sentences can give a non-zero contribution to the mixed moment in \eqref{eq: mixed mom in terms of sentence}.

\begin{definition}[CLT sentence]\label{def: CLT sentence}
Let $a=[w_{i}]_{i=1}^{m}$ be a weak CLT sentence consisting of $m$ words with length $l_{1}+1, \ldots, l_{m}+1$ respectively. Then $a$ is called a CLT sentence if $|V(G_a)| = \frac{\sum_{r=1}^{m} l_{r}}{2}$
\end{definition}
We will see that CLT sentences give the main contribution to mixed moment \eqref{eq: mixed mom in terms of sentence}, since the entropy of any sentence $a$ with smaller $|V(G_a)|$ will be of lower order in $N$.

The following proposition will be used to show that the sum in \eqref{eq: mixed mom in terms of sentence} restricted to CLT sentences factors in way that makes Wick's formula hold in the limit $N\to\infty$.

\begin{proposition}\label{prop:az}(Proposition 4.9 in \cite{AZ05})
Let $a=[w_{i}]_{i=1}^{m}$ be a weak CLT sentence consisting of $m$ words with length $l_{1}+1, \ldots, l_{m}+1$ respectively. Then we have
$|V(G_a)| \le \frac{\sum_{r=1}^{m} l_{r}}{2}$. Suppose equality occurs (i.e. $a$ is a CLT sentence), then the words $w_{i}$ of the sentence $a$ are perfectly paired
in the sense that for all $i$ there exists a unique $j$ distinct from $i$ such that $w_{i}$ and $w_j$
have a letter in common. In particular, $m$ is even.
\end{proposition}

Using this proposition and \eqref{eq: word weight indep} we have
for any CLT sentence $a$ with pairing $\eta(a)$ that
\begin{equation}\label{eq: CLT sent pairing indep}
\EN\left[\hat{J}_{a}\right]=\EN\left[\prod_{\left\{ i,j\right\} \in\eta\left(a\right)}\hat{J}_{a_{i}}\hat{J}_{a_{j}}\right]=\prod_{\left\{ i,j\right\} \in\eta\left(a\right)}\EN\left[\hat{J}_{a_{i}}\hat{J}_{a_{j}}\right].
\end{equation}

We have now stated all the ingredients necessary to prove Proposition \ref{prop: cycle counts under P}.
\begin{proof}[Proof of Proposition \ref{prop: cycle counts under P}]

Consider the vector $\left(C_{N,1},\ldots,C_{N,k}\right)$. We will
show that the mixed moments of this vector converge and that the limit
is a sum over pairings, to be able to apply Lemmas \ref{lem:mom} and \ref{lem:wick}.

To this end consider for any $m\ge1,l_{1},\ldots,l_{m}\in\left\{ 1,\ldots,k\right\}$ the mixed moment
\[
\EN\left[X_{N,1}\ldots X_{N,m}\right],
\]
where $X_{N,i}=C_{N,l_{i}}$.  By \eqref{eq: cycle count in terms of cyclic words} this equals
\[
\begin{array}{l}
\frac{1}{N^{\frac{l_{1}+\ldots+l_{m}}{2}}} \displaystyle{\sum_{w_{1}\in\mathfrak{W}_{l_{1}+1},\ldots, w_{m}\in\mathfrak{W}_{l_{m}+1}}}\EN\left[\hat{J}_{w_{1}}\ldots\hat{J}_{w_{m}}\right],\end{array}
\]
which by \eqref{eq: zero if not weak clt} is the same as
\[
\frac{1}{N^{\frac{l_{1}+\ldots+l_{m}}{2}}}\sum_{a\in\mathfrak{W}_{l_{1}+1}\times\ldots\times\mathfrak{W}_{l_{m}+1}:a\text{ weak CLT sequence}}\EN\left[\hat{J}_{a}\right].
\]
Consider the magnitude of the contribution of weak CLT sentences whose
corresponding graph has a less than maximal number of vertices, i.e.
\begin{equation}
\left|\frac{1}{N^{\frac{l_{1}+\ldots+l_{m}}{2}}}\sum_{a\in\mathfrak{W}_{l_{1}+1}\times\ldots\times\mathfrak{W}_{l_{m}+1}:a\text{ weak CLT sequence},\left|V\left(G_{a}\right)\right|<\frac{l_{1}+\ldots+l_{m}}{2}}\EN\left[\hat{J}_{a}\right]\right|\label{eq: non CLT sequences-1}.
\end{equation}
Note that the moment $\EN\left[\hat{J}_{a}\right]$ depends only on the \emph{equivalence class of $a$}. Furthermore the set of equivalence classes is a function only of $m,l_1,\ldots,l_m$ and \emph{not} of $N$. Thus we obtain that
\[
\sup_{a\in\mathfrak{W}_{l_{1}+1}\times\ldots\times\mathfrak{W}_{l_{m}+1}}\EN\left[\hat{J}_{a}\right]\le M,
\]
where $M<\infty$ does not depend on $N$, and (\ref{eq: non CLT sequences-1})
is bounded above by 
\begin{equation}\label{eq: sum fewer vertices}
\frac{M}{N^{\frac{l_{1}+\ldots+l_{m}}{2}}}\sum_{a\in\mathfrak{W}_{l_{1}+1}\times\ldots\times\mathfrak{W}_{l_{m}+1}:a\text{ weak CLT sequence},\left|V\left(G_{a}\right)\right|<\frac{l_{1}+\ldots+l_{m}}{2}}1.
\end{equation}
Noting that the number of vertices of $G_a$ depends only on the equivalence class of $a$, and that the number $K$ of equivalence classes of sentences such that $|V(G_a)| < \frac{l_1+\ldots+l_m}{2}$ does not depend on $N$, we can crudely bound the sum in \eqref{eq: sum fewer vertices} by $K\times N^{\frac{l_{1}+\ldots+l_{m}}{2}-1}$. Thus (\ref{eq: non CLT sequences-1})
is at most $cN^{-1}$ for a constant $c$ that depends only on $m,l_{1},\ldots,l_{m}$, and so
(recalling Definition \ref{def: CLT sentence})
\[
\left|\EN\left[X_{N,1}\ldots X_{N,m}\right]-\frac{1}{N^{\frac{l_{1}+\ldots+l_{m}}{2}}}\sum_{a\in\mathfrak{W}_{l_{1}+1}\times\ldots\times\mathfrak{W}_{l_{m}+1}:a\text{ is a CLT sequence}}\EN\left[\hat{J}_{a}\right]\right|\to0.
\]
Now by Proposition \ref{prop:az} all CLT sentences $a$ have length $m$
that is even, which in particular implies that $\EN\left[X_{N,1}\ldots X_{N,m}\right]\to0$
if $m$ is odd. Furthermore recalling \eqref{eq: CLT sent pairing indep} we have
\[
\displaystyle{\sum_{a\in\mathfrak{W}_{l_{1}+1}\times\ldots\times\mathfrak{W}_{l_{m}+1}:a\text{ is a CLT sequence}}}\EN\left[\hat{J}_{a}\right]
=\displaystyle{\sum_{\eta}\sum_{a\in \mathcal{C}_\eta}\prod_{\left\{ i,j\right\} \in\eta}}\EN\left[\hat{J}_{a_{i}}\hat{J}_{a_{j}}\right],
\]
where the sum over $\eta$ is over all pairings of $\left\{ 1,\ldots,m\right\} $ and with $C_{\eta}$ denoting the
set of all $a\in\mathfrak{W}_{l_{1}+1}\times\ldots\times\mathfrak{W}_{l_{m}+1}$
that are CLT sentences whose pairing $\eta(a)$ satisfies $\eta\left(a\right)=\eta$. Now note that for any fixed $\eta$  we have using \eqref{eq: word weight covar}
\begin{equation}
\begin{array}{l}
\frac{1}{N^{\frac{l_{1}+\ldots+l_{m}}{2}}}\left|{\displaystyle \sum_{a\in\mathcal{C}_{\eta}}}\prod_{\left\{ i,j\right\} \in\eta}\EN\left[\hat{J}_{a_{i}}\hat{J}_{a_{j}}\right]-{\displaystyle \sum_{a\in\mathfrak{W}_{l_{1}+1}\times\ldots\times\mathfrak{W}_{l_{m}+1}}}\prod_{\left\{ i,j\right\} \in\eta}\EN\left[\hat{J}_{a_{i}}\hat{J}_{a_{j}}\right]\right|\\
\le\frac{cM}{N^{\frac{l_{1}+\ldots+l_{m}}{2}}}\left|\displaystyle{\sum_{a\in\mathcal{C}_{\eta}}}1_{\left\{ G_{a_{i}}=G_{a_{j}}\text{\,for all }\left\{ i,j\right\} \in\eta\right\} }-{\displaystyle \sum_{a\in\mathfrak{W}_{l_{1}+1}\times\ldots\times\mathfrak{W}_{l_{m}+1}}}1_{\left\{ G_{a_{i}}=G_{a_{j}}\text{\,for all }\left\{ i,j\right\} \in\eta\right\} }\right|.
\end{array}\label{eq: CLT sent diff}
\end{equation}
All sentences $a$ that are CLT sentences
appear in both sums on the bottom line, while the only sentences that appears only in one
are those for which $G_{a_{i}}=G_{a_{j}}$ for all $\left\{ i,j\right\} \in \eta$
but $G_{a_{i}}$ shares a vertex with $G_{a_{j}}$ for some $\left\{ i,j\right\} \notin \eta $.
If so $G_{a}$ must necessarily have less than $\frac{l_{1}+\ldots+l_{m}}{2}$
vertices. Therefore the difference in the bottom line of \eqref{eq: CLT sent diff} is bounded above by
$ K N^{\frac{l_{1}+\ldots+l_{m}}{2}-1}$, for the constant $K$ from before, so that (\ref{eq: CLT sent diff}) goes to zero as $N\to\infty$. Since also the number of pairings $\eta$ does not depend on $N$ we have
\[
\left|\EN\left[X_{N,1}\ldots X_{N,m}\right]-\sum_{\eta}\frac{1}{N^{\frac{l_{1}+\ldots+l_{m}}{2}}}\sum_{a\in\mathfrak{W}_{l_{1}+1}\times\ldots\times\mathfrak{W}_{l_{m}+1}}\prod_{\left\{ i,j\right\} \in\eta}\EN\left[\hat{J}_{a_{i}}\hat{J}_{a_{j}}\right]\right|\to0.
\]
Finally the last sum over $a$ factors as
\[
\begin{array}{lcl}
\frac{1}{N^{\frac{l_{1}+\ldots+l_{m}}{2}}}\displaystyle{\sum_{a\in\mathfrak{W}_{l_{1}+1}\times\ldots\times\mathfrak{W}_{l_{m}+1}}}\prod_{\left\{ i,j\right\} \in\eta}\EN\left[\hat{J}_{a_{i}}\hat{J}_{a_{j}}\right]
&=&\displaystyle{\prod_{\left\{ i,j\right\} \in\eta}}\frac{1}{N^{\frac{l_{i}+l_{j}}{2}}}
\displaystyle{\sum_{w\in\mathfrak{W}_{l_{i}+1},v\in\mathfrak{W}_{l_{j}+1}}}\EN\left[\hat{J}_{w}\hat{J}_{v}\right]\\
&=&\displaystyle{\prod_{\left\{ i,j\right\} \in\eta}}\EN\left[\frac{\sum_{w\in\mathfrak{W}_{l_{i}+1}}\hat{J}_{w}}{N^{\frac{l_{i}}{2}}}\frac{\sum_{v\in\mathfrak{W}_{l_{j}+1}}\hat{J}_{v}}{N^{\frac{l_{j}}{2}}}\right]\\
&\overset{\eqref{eq: cycle count in terms of cyclic words}}{=}&
\displaystyle{\prod_{\left\{ i,j\right\}\in\eta}}\EN\left[C_{N,i}C_{N,j}\right].
\end{array}
\]
Applying also Lemma \ref{lem: means and variances} we have showed that
$$\EN[ X_{N,1}\ldots X_{N,m} ] \to \sum_\eta \prod_{\{i,j\}\in\eta} \E[C_{\infty,i}C_{\infty,j}],$$
i.e. that all the mixed moments converge. Applying this with $m=2h$ and $X_i = C_{N,l}$, and \eqref{eqn:wick} with $X_i=C_{\infty,l}$, one sees that $\lim_{N\to\infty} \EN[ C_{N,l}^{2h}] = \E[ C_{\infty,l}^{2h} ]$ for all positive integers $h$, so that Carleman's condition is easily verified. Therefore by Lemma \ref{lem:mom} the vector
$\left(C_{N,1},\ldots,C_{N,k}\right)$ converges in distribution
to a random vector $\left(Z_{1},\ldots,Z_{k}\right)$. By the same
lemma we have $\EN\left[C_{N,i}C_{N,j}\right]\to \E\left[Z_{i}Z_{j}\right]$,
so that in fact $\E\left[Z_{i}Z_{j}\right]=\E\left[C_{\infty,i}C_{\infty,j}\right]$
for all $i,j$, which means that for all $m\ge1,1\le l_{1},\ldots,l_{m}\le k$ we have
with $X_{j}=Z_{l_{j}}$ that
\[
\E\left[X_{1}\ldots X_{m}\right]=\begin{cases}
\sum_{\eta}\E\left[X_{i}X_{j}\right] & \text{ if }m\text{\,is even,}\\
0 & \text{ otherwise.}
\end{cases}
\]
Thus $\left(Z_{1},\ldots,Z_{k}\right)$ satisfies Wicks formula \eqref{eqn:wick} so by Lemma \ref{lem:wick} the vector
$\left(Z_{1},\ldots,Z_{k}\right)$ is Gaussian. Since its
covariance matches that of $\left(C_{\infty,1},\ldots,C_{\infty,k}\right)$
also its law does.
\end{proof}
\section{Fluctuations of partition function determined by cycle counts; Proof of Theorem \ref{thm: sec thm}}
In this section we will prove Theorem \ref{thm: sec thm}.
Denote the normalized partition function by
\begin{equation}\label{eq: Z hat def}
\hat{Z}_{N}=\frac{Z_{N}}{\EN\left[Z_{N}\right]},
\end{equation}
so that
\begin{equation}\label{eq: Z hat mean}
    \EN[ \hat{Z}_N] = 1.
\end{equation}
Since also $\hat{Z}_{N}\ge0$
we can use it to define a tilted measure $\mathbb{Q}_{N}$ via 
\[
\frac{d\mathbb{Q}_{N}}{d\mathbb{P}_{N}}=\hat{Z}_{N}.
\]
Let 
$ \mathbb{Q}_{N,\sigma},$ 
be a measure under which the $J_{ij}$ are Gaussian with the same covariance as under $\PN$ (see \eqref{eq: J tilde law}), but where $J_{ij}$ has mean $\frac{1}{\sqrt{N}}\beta\sqrt{2\alpha_2} \sigma_i \sigma_j$. Note that
\begin{equation}\label{eq: QNsigma def}
    \frac{d\mathbb{Q}_{N,\sigma}}{d\mathbb{P}_N} = \exp\left( \sum_{i} \left\{\frac{\beta \sqrt{\alpha_2} J_{ii}}{\sqrt{2N}}-\frac{\beta^2\alpha_2}{2N}\right\} +\sum_{i<j} \left\{ \frac{\beta\sqrt{2\alpha_2}\sigma_{i}\sigma_{j}J_{ij}}{\sqrt{N}} - \frac{\beta^2\alpha_2}{N} \right\} \right).
\end{equation}    
We then have the following, which can be interpreted as saying that under $\mathbb{Q}_N$ the $J_{ij}$ have the law of a mixture of Gaussians. Namely conditionally on $\tilde{H}_N$ one samples $\sigma$ according to the Gibbs measure of $\tilde{H}_N$, and then samples $J_{ij}$ as Gaussians with mean $\frac{1}{\sqrt{N}}\beta\sqrt{2\alpha_2} \sigma_i \sigma_j$ and the covariance of $J_{ij}$ under $\PN$.
\begin{lemma}[Radon-Nikodym derivative identity]\label{lem: rad niko}
It holds that 
\begin{equation}
    \begin{split}
        &\frac{d\mathbb{Q}_{N}}{d\mathbb{P}_{N}}= \sum_{\sigma \in \{-1 , +1 \}^{N}} \frac{1}{2^{N}} \frac{\exp\left(\beta\tilde{H}_{N}\left(\sigma\right)\right)}{\mathbb{E}_{N}\left[\exp\left(\beta\tilde{H}_{N}\left(\sigma\right)\right)\right]  }\frac{d\mathbb{Q}_{N,\sigma}}{d\mathbb{P}_N}
    \end{split}
\end{equation}
\end{lemma}

\begin{proof}
Since $H_N(\sigma) = \sqrt{\alpha_2} H_N^2(\sigma) + \tilde{H}_N(\sigma)$ and $H_N^2$ and $\tilde{H}_N$ are independent we obtain that
$$ \frac{d\mathbb{Q}_{N}}{d\mathbb{P}_{N}} = \frac{E[\exp\left(\beta H_N(\sigma) \right)]}{\EN[\exp\left(\beta H_N(\sigma) \right)]} = \sum_{\sigma} \frac{1}{2^N} \frac{\exp\left( \beta \sqrt{\alpha_2} H_N^2(\sigma) \right)}{\EN[ \exp\left( \beta \sqrt{\alpha_2} H_N^2(\sigma) \right)]}
\frac{\exp\left(  \beta \tilde{H}_N(\sigma) \right)}{\EN\left[ \exp\left( \beta \tilde{H}_{N}(\sigma)\right) \right]}.$$
From the definition \eqref{eq: HN2 def} of $H_N^2$ and that $\Var[H^2_N(\sigma)] = 1$ it follows that
\begin{equation} 
\frac{\exp\left( \beta \sqrt{\alpha_2} H^2_N(\sigma) \right)}{\EN\left[ \exp\left(\beta \sqrt{\alpha_2} H_{N}^{2}(\sigma)\right) \right]}\\ 
= \frac{
    \exp\left(\frac{\beta\sqrt{\alpha_{2}}}{\sqrt{2N}}\sum_{i}J_{ii}+\frac{\beta\sqrt{2\alpha_{2}}}{\sqrt{N}}\sum_{i<j}J_{ij}\sigma_{i}\sigma_{j}\right)
}{
    \exp\left(\frac{\beta^{2}\alpha_{2}}{2}N\right)
},
\end{equation}
which equals $\frac{d\mathbb{Q}_{N,\sigma}}{d\mathbb{P}_N}$ by \eqref{eq: QNsigma def}.
\end{proof}
The following is an easy consequence of the lemma and the independence of $\tilde{H}_N$ and $\frac{d\mathbb{Q}_{N,\sigma}}{d \PN}$.
\begin{corollary}\label{cor: QN mixture}
We have
\begin{equation}\label{eq: QN mixture}
    \mathbb{Q}_N[A] = E[ \mathbb{Q}_{N,\sigma}[A]]\text{ for any event }A,
\end{equation}
measurable with respect to the $J_{ij}$.
\end{corollary}

We will need a variant of Proposition \ref{prop: cycle counts under P}
for the law of the cycle counts under the measure $\mathbb{Q}_{N}$,
namely the following.

\begin{proposition}[Limiting law of centered cycle counts under $\mathbb{Q}_N$]
\label{prop: cycle counts under Q} Suppose that $\beta \sqrt{2\alpha_2}\le 1$. For any $k\ge1$ it holds that
\begin{equation}\label{eq: cycle counts under Q}
\mathbb{Q}_{N}\text{-law of }\left(C_{N,1},C_{N,2},\ldots,C_{N,k}\right)\overset{D}{\to}\left(\mu_1+C_{\infty,1},\ldots,\mu_k+C_{\infty,k}\right),
\end{equation}
for $\left(C_{\infty,1},\ldots,C_{\infty,k}\right)$ as in Proposition \ref{prop: cycle counts under P} and $\mu_{k}$ as in \eqref{eq: muk def}.
\end{proposition}

\begin{proof}

By \eqref{eq: cycle count in terms of cylic word non cent weight} we have for all $l\ge 1$
\begin{equation}\label{eq:expandcyclealt}
\begin{split}
 C_{N,{l}} & =   \frac{1}{N^{\frac{l}{2}}}\sum_{w \in \mathfrak{W}_{l+1}} \prod_{j=0}^{l-1}J_{w_jw_{j+1}}- (N-1)1_{\{l=2\}}\\
   & =   \frac{1}{N^{\frac{l}{2}}}\sum_{w \in \mathfrak{W}_{l+1}} \prod_{j=0}^{l-1} \left( J_{w_{j}w_{j+1}}- \sigma_{w_{j}w_{j+1}} + \sigma_{w_{j}w_{j+1}} \right) - (N-1)1_{\{l=2\}}.
\end{split}
\end{equation}
where we use the shorthand
\[
\sigma_{ij}=\frac{\beta\sqrt{2\alpha_2}}{\sqrt{N}}\sigma_{i}\sigma_{j}.
\]
Recall from above \eqref{eq: QNsigma def} that under $\mathbb{Q}_{N,\sigma}$ the $J_{ij}- \sigma_{ij}$ have the same law as the $J_{ij}$ under $\mathbb{P}_{N}$. Now letting $B_{ij},i \le j$ have this same law under an auxilliary probability $\mathbb{P}$ and letting $\sigma_i$ be IID Rademacher random variables under $\mathbb{P}$, independent also of the $B_{ij}$, we have from \eqref{eq: QN mixture} that
\begin{equation}\label{eq: QN law}
\mathbb{Q}_{N}-\text{law of } C_{N,l} = \mathbb{P} -\text{law of } \frac{1}{N^{\frac{l}{2}}}\sum_{w \in \mathfrak{W}_{l+1}} \prod_{j=0}^{l-1}\left( B_{w_{j}w_{j+1}} + \sigma_{w_jw_{j+1}}\right) - (N-1)1_{\{l=2\}}.
\end{equation}
If $l=2$ then the RHS can be written as
\begin{equation}\label{eq: l2 form}
    \frac{1}{N^{\frac{l}{2}}}\sum_{w\in\mathfrak{W}_{l+1}}\prod_{j=0}^{l-1}B_{w_{j}w_{j+1}}-\left(N-1\right)+\frac{2}{N}\sum_{i\neq j}B_{ij}\sigma_{ij}+\frac{1}{N^{\frac{l}{2}}}\sum_{w\in\mathfrak{W}_{l+1}}\prod_{j=0}^{l-1}\sigma_{w_{j}w_{j+1}}.
\end{equation}
If $l\ne 2$ the product
$\prod_{j=0}^{l-1}\left(B_{w_{j}w_{j+1}}+\sigma_{w_{j}w_{j+1}}\right)$
can be written as a sum over subgraphs of $G_{w}$, namely
\[
\sum_{H\subset G_{w}}\left(\prod_{e\in E\left(G_{w}\backslash H\right)}B_{e}\right)\left(\prod_{\sigma\in E\left(H\right)}\sigma_{e}\right),
\]
where $G_{w}\backslash H$ denotes the graph on $\left\{ 1,2\ldots,N\right\} $
with vertex set $V(G_w)$ and edge set $E\left(G_{w}\right)\backslash E\left(H\right)$. Thus the quantity on the right-hand side of \eqref{eq: QN law} equals
\begin{equation}\label{eq: lgen form}
\frac{1}{N^{\frac{l}{2}}}\sum_{w\in\mathfrak{W}_{l+1}}\prod_{j=0}^{l-1}B_{w_{j}w_{j+1}}
+
\sum_{w\in\mathfrak{W}_{l+1}}V_{N,l,w}
+
\frac{1}{N^{\frac{l}{2}}}\sum_{w\in\mathfrak{W}_{l+1}}\prod_{j=0}^{l-1}\sigma_{w_{j}w_{j+1}},
\end{equation}
for
\[
V_{N,l,w}={ \frac{1}{N^{\frac{l}{2}}}}
\sum_{\emptyset\ne H \subsetneq G_{w}}
\left(\prod_{e\in E\left(G_{w}\backslash H\right)}B_{e}\right)
\left(\prod_{e\in E\left(H\right)}\sigma_{e}\right).
\]
Observe that  $B_{N,l}:=\frac{1}{N^{\frac{l}{2}}}\sum_{w\in\mathfrak{W}_{l+1}}\prod_{j=0}^{l-1}B_{w_{j}w_{j+1}} - (N-1)1_{\{l=2\}}$ have exactly same joint distribution under $\mathbb{P}$ as the $C_{N,l}$ do under $\mathbb{P}_{N}$. Hence from Proposition \ref{prop: cycle counts under P} we have that
\begin{equation}\label{eq: conv of Bs}
    \mathbb{P}\text{-law of }\left(B_{N,1},B_{N,2},B_{N,3}\ldots, B_{N,k} \right)\stackrel{D}{\to} \left( C_{\infty,1},\ldots, C_{\infty,k} \right).
\end{equation}
On the other hand we have that
\begin{equation}\label{eq: conv of means}
\frac{1}{N^{\frac{l}{2}}}\sum_{w\in\mathfrak{W}_{l+1}}\prod_{j=0}^{l-1}\sigma_{w_{j}w_{j+1}}=\frac{1}{N^{\frac{l}{2}}}\sum_{w\in\mathfrak{W}_{l+1}}\left(\frac{\beta\sqrt{2\alpha_{2}}}{\sqrt{N}}\right)^{l}\overset{\eqref{eq: muk def},\eqref{eq: size of Wk}}{=}\left(1+o\left(1\right)\right)\mu_{l}.
\end{equation}
A variance calculation shows that
\begin{equation}\label{eq: to show when k is 2}
\frac{1}{N}\sum_{i\neq j}B_{ij}\sigma_{ij} = \frac{\beta \sqrt{2\alpha_{2}}}{N^{\frac{3}{2}}} \sum_{i \neq j} B_{ij}\sigma_{i}\sigma_{j} \overset{P}{\to} 0\text{ under }\mathbb{P}.
\end{equation}
In the remainder we will show that
\begin{equation}\label{eq: V to show}
\sum_{w\in\mathfrak{W}_{l+1}}V_{N,l,w} \overset{P}{\to} 0\text{ under }\mathbb{P} \text{ for } 3 \le l \le k.
\end{equation}
Recalling \eqref{eq: QN law}, \eqref{eq: l2 form} and \eqref{eq: lgen form} the claim \eqref{eq: cycle counts under Q} follows from \eqref{eq: conv of Bs}-\eqref{eq: V to show} (note that $V_{N,1,w}=0$) and Slutsky's theorem.

It thus only remains to prove \eqref{eq: V to show}. To this end we compute
$$ \mathbb{E}\left[\left(\sum_{w\in\mathfrak{W}_{l+1}}V_{N,l,w}\right)^{2}\right] = \sum_{w,v\in\mathfrak{W}_{l+1}}\mathbb{E}\left[V_{N,l,w}V_{N,l,v}\right]. $$
This second moment equals
$$
\frac{1}{N^{l}}\sum_{w,v\in\mathfrak{W}_{l+1}}\sum_{\emptyset\ne H\subsetneq G_{w},\emptyset\ne H'\subsetneq G_{v}}\left(\prod_{e\in E\left(H\right)\cup E\left(H'\right)}\sigma_{e}\right)\mathbb{E}\left[\left(\prod_{e\in E\left(G_{w}\backslash H\right)}B_{e}\right)\left(\prod_{e\in E\left(G_{v}\backslash H'\right)}B_{e}\right)\right].
$$
We have 
\[
\left|\prod_{e\in E\left(H\right)\cup E\left(H'\right)}\sigma_{e}\right|\le\left(\frac{\beta\sqrt{2\alpha_{2}}}{\sqrt{N}}\right)^{\left|E\left(H\right)\right|+\left|E\left(H'\right)\right|}\le N^{-\frac{\left|E\left(H\right)\right|+\left|E\left(H'\right)\right|}{2}},
\]
(since $\beta\sqrt{2\alpha_{2}}\le1$) and since the $B_{e}$
are independent and have mean zero
\[
\mathbb{E}\left[\left(\prod_{e\in E\left(G_{w}\backslash H\right)}B_{e}\right)\left(\prod_{e\in E\left(G_{v}\backslash H'\right)}B_{e}\right)\right]=\begin{cases}
1 & \text{\,if }G_{w}\backslash H=G_{v}\backslash H',\\
0 & \text{else}.
\end{cases}
\]
Thus
\[
\mathbb{E}\left[\left(\sum_{w\in\mathfrak{W}_{l+1}}V_{N,l,w}\right)^{2}\right]  \le  \frac{1}{N^{l}}\sum_{w,v\in\mathfrak{W}_{l+1}}\sum_{\emptyset\ne H\subsetneq G_{w},\emptyset\ne H'\subsetneq G_{v}}1_{\left\{ G_{w}\backslash H=G_{v}\backslash H'\right\} 
}N^{-\frac{\left|E\left(H\right)\right|+\left|E\left(H'\right)\right|}{2}}\\
\]
Since $ G_{w}\setminus H=G_{v}\setminus H'$ implies $|E(H)|=|E(H')|$, and $1_{\left\{ G_{w}\setminus H=G_{v}\setminus H'\right\} }\le1_{\left\{ G_{w}\setminus H\subsetneq G_{v}\right\} }$
and there are most $2^{l}$ ways to choose $H'$, this is at
most
\[
\frac{2^{l}}{N^{l}}\sum_{w\in\mathfrak{W}_{l+1}}\sum_{\emptyset\ne H\subsetneq G_{w}}N^{-\left|E\left(H\right)\right|}\sum_{v\in\mathfrak{W}_{l+1}}1_{\left\{ G_{w}\backslash H\subsetneq G_{v}\right\}}
\]
We now bound the sum over $v$ combinatorially. For a pair $(w,H)$ such that $w=\left(w_{0},w_{1},\ldots,w_{l}\right)\in\mathfrak{W}_{l+1}$ and $H\subset G_{w}$
we encode $H$ as a vector $h\left(w,H\right)\in\left\{ 0,1\right\} ^{l}$
by setting $h_{i}\left(w,H\right)=1$ if $\left\{ w_{i},w_{i+1}\right\} \in E\left(H\right)$.
Next define an equivalence relation for triples $\left(w,h,v\right)\in\mathfrak{W}_{l+1}\times\left\{ 0,1\right\} ^{l}\times\mathfrak{W}_{l+1}$
under which $\left(w,h,v\right)$ is equivalent to $\left(w',h',v'\right)$
if $h=h'$ and $\left(w,v\right)\sim\left(w',v'\right)$ in the equivalence
relation for sentences. Let $\left[\left(w,h,v\right)\right]$ denote
the equivalence class of $\left(w,h,v\right)$ and let $\mathcal{E}$
denote the set of all equivalence classes.

The indicator $1_{\left\{ G_{w}\backslash H\subsetneq G_{v}\right\} }$
is a function only of the equivalence class $e=\left[\left(w,h\left(w,H\right),v\right)\right]$.
Denote this function by $f\left(e\right)$. Using these constructions
one can write

\[
\sum_{v\in\mathfrak{W}_{l+1}}1_{\left\{ G_{w}\backslash H\subsetneq G_{v}\right\} }=\sum_{e\in\mathcal{E}}\left|\left\{ v\in\mathfrak{W}_{l+1}:\left[\left(w,h\left(w,H\right),v\right)\right]=e\right\} \right|f\left(e\right).
\]
Now note the following:
\begin{itemize}
\item The number of vertices in $G_{v}\backslash\left(G_{w}\backslash H\right)$
depends only on the equivalence class $e=\left[\left(w,h\left(w,H\right),v\right)\right]$.
Denote this number of vertices by $A\left(e\right)$. 
\item For given $w,H$ and $e\in\mathcal{E}$, to construct a $v$ such
that $\left[\left(w,h\left(w,H\right),v\right)\right]=e$ one must
pick indices in $\left\{ 1,\ldots,N\right\} $ for $A\left(e\right)$
vertices (the indices of the other vertices are fixed by $w$), so
that
\[
\left|\left\{ v\in\mathfrak{W}_{l+1}:\left[\left(w,h\left(w,H\right),v\right)\right]=e\right\} \right|\le N^{A\left(e\right)}.
\]
\item For all $w,H,v$ it holds that $A\left(\left[\left(w,h\left(w,H\right),v\right)\right]\right)\le\left|E\left(G_{v}\right)\right|-\left|E\left(G_{w}\backslash H\right)\right|-1=\left|E\left(G_{w}\right)\right|-\left|E\left(G_{w}\backslash H\right)\right|-1=\left|E\left(H\right)\right|-1$,
where we have equality in the bound if $H$ if $G_{w}\setminus H$
is a line, and one uses that $H\ne\emptyset,E\left(G_{w}\right)$.
\item The number of equivalence classes $\left|\mathcal{E}\right|$ is finite and
independent of $N$. 
\end{itemize}
With these facts we get that
\[
\sum_{v\in\mathfrak{W}_{l+1}}1_{\left\{ G_{w}\backslash H\subsetneq G_{v}\right\} }\le\left|\mathcal{E}\right|N^{\left|E\left(H\right)\right|-1}.
\]

Thus
\[
\begin{array}{ccl}
\mathbb{E}\left[\left(\displaystyle{\sum_{w\in\mathfrak{W}_{l+1}}}V_{N,l,w}\right)^{2}\right] & \le & \displaystyle{\frac{2^{l}|\mathcal{E}|}{N^{l}}}\displaystyle{\sum_{w\in\mathfrak{W}_{l+1}}\sum_{\emptyset\ne H\subsetneq G_{w}}}\frac{1}{N}\le\frac{4^{l}|\mathcal{E}|}{N^{l+1}}\left|\mathfrak{W}_{l+1}\right|\le\frac{4^{l}|\mathcal{E}|}{N},\end{array}
\]
since $\left|\mathfrak{W}_{l+1}\right|\le N^{l}$ by \eqref{eq: size of Wk}. This proves \eqref{eq: V to show} and thus concludes the proof of the proposition.
\end{proof}

We will also need a precise asymptotic for the second moment of $\hat{Z}_{N}$ under $\mathbb{P}_{N}$.
\begin{lemma}[Asymptotic for second moment of partition function]
\label{lem: sec mom}For every $\xi$ there is a $\beta_{\xi} \in (0, \frac{1}{\sqrt{2 \alpha_2}}]$ such
that if $0\le\beta<\beta_{\xi}$ then
\[
\mathbb{E}_{N}\left[\hat{Z}_{N}^{2}\right]\to\frac{1}{\sqrt{1-2 \alpha_2 \beta^2}},\text{ as }N\to\infty.
\]
If $\xi(x)=\alpha_2 x^2$ then $\beta_{\xi} = \frac{1}{\sqrt{2\alpha_2}}$.
\end{lemma}
\begin{proof}
We have
\begin{equation}
\begin{array}{ccl}
\EN\left[\hat{Z}_{N}\right] & \overset{\eqref{eq: Z hat def}}{=} & \frac{1}{\EN\left[Z_{N}\right]^{2}}\EN\left[Z_{N}^{2}\right]\\
 & = & \frac{1}{\EN\left[Z_{N}\right]^{2}}E^{\otimes2}\left[\EN\left[\exp\left(\beta H_{N}(\sigma)+\beta H_{N}(\sigma')\right)\right]\right]\\
 & \overset{\eqref{eq: hamilt covar}}{=} & \frac{1}{\EN\left[Z_{N}\right]^{2}}E^{\otimes2}\left[\exp\left(\beta^{2}\xi\left(1\right)N+\beta^{2}N\xi\left(\frac{\sigma\cdot\sigma'}{N}\right)\right)\right]\\
 & \overset{\eqref{eq: ann Z}}{=} & E\left[\exp\left(\beta^{2}N\xi\left(\frac{\sum_{i=1}^{N}\sigma_{i}}{N}\right)\right)\right].
\end{array}
\end{equation}
Let $I\left(\alpha\right)=\frac{1+\alpha}{2}\log\left(1+\alpha\right)+\frac{1-\alpha}{2}\log\left(1-\alpha\right)$
be the large deviation rate function of the sum $\sum_{i=1}^{N}\sigma_{i}$ under $E$,
and let
\[
\beta_{\xi}=\sup\left\{ \beta:\beta^{2}\xi\left(\alpha\right)-I\left(\alpha\right)\text{\,has a unique global maximum in $[-1,1]$ at }\alpha=0\right\} .
\]
Since $I\left(\alpha\right)$ has its unique global maximum at $\alpha=0$
and $I''\left(0\right)=-1$ and $\xi'\left(0\right)=0$ it follows
by continuity that $\beta_{\xi}>0$. Also by considering the second
derivative at $0$ we see that $\beta_{\xi}\le\sqrt{2 \alpha_2}.$ Furthermore since $I''(m)=\frac{1}{1-m^2}$ we have when $\xi(x)=\alpha_2 x^2$ that $\alpha\to\beta^2\xi(\alpha)-I(\alpha)$ is concave for $\beta < \sqrt{2 \alpha_2}$, so then $\beta_{\xi}=\sqrt{2\alpha_2}$.

Now assume $0\le\beta<\beta_{\xi}$. Let $\varepsilon>0$. We have that 
\begin{equation}\label{eq: split in 3}
\begin{array}{rcl}
E\left[\exp\left(\beta^{2}N\xi\left(\frac{\sum_{i=1}^{N}\sigma_{i}}{N}\right)\right)\right]
&= & 	E\left[\exp\left(\beta^{2}N\xi\left(\frac{\sum_{i=1}^{N}\sigma_{i}}{N}\right)\right)1_{\left\{ \left|\frac{\sum_{i=1}^{N}\sigma_{i}}{\sqrt{N}}\right|\le M\right\} }\right]\\
&& +E\left[\exp\left(\beta^{2}N\xi\left(\frac{\sum_{i=1}^{N}\sigma_{i}}{N}\right)\right)1_{\left\{ M\le\left|\frac{\sum_{i=1}^{N}\sigma_{i}}{\sqrt{N}}\right|\le\varepsilon\sqrt{N}\right\} }\right]\\
&& +E\left[\exp\left(\beta^{2}N\xi\left(\frac{\sum_{i=1}^{N}\sigma_{i}}{N}\right)\right)1_{\left\{ \left| \frac{\sum_{i=1}^{N}\sigma_{i}}{N} \right| \ge\varepsilon\right\} }\right].
\end{array}
\end{equation}
Note that
\begin{equation}
\xi\left(x\right)=\alpha_2 x^2+O_{\xi}\left(x^{3}\right).\label{eq: xi expansion}
\end{equation}
Using this and Bernstein's inequality and (\ref{eq: xi expansion}) the second
line on the right-hand side is bounded by
\[
\begin{array}{l}
\sum_{M\le m\le\varepsilon\sqrt{N}}
\exp\left(
\beta^{2}\alpha_{2}m^{2}+c_{\xi}\frac{m^{3}}{\sqrt{N}}\right)\exp\left(-\frac{m^{2}}{2\left(1+\frac{m}{\sqrt{N}}\right)}
\right)\\
\le \sum_{M \le m \le \varepsilon \sqrt{N}} \exp\left(-m^2\left( \frac{1}{2(1+\varepsilon)} - c_{\xi}\varepsilon + \beta^2 \alpha_{2} \right)  \right) = o_{M}\left(1\right),
\end{array}
\]
provided $\varepsilon$ is chosen small enough depending on $\xi,\beta,\alpha_2$. %
By a large deviation bound the last line of \eqref{eq: split in 3} is bounded by
\[
\begin{array}{l}
\sum_{l:\varepsilon\le\frac{l}{N}\le1}\exp\left(N\left(\beta^{2}\xi\left(\frac{l}{N}\right)-I\left(\frac{l}{N}\right)+o(1)\right)\right)\\
\le N\exp\left(N\sup_{\left|\alpha\right|\ge\varepsilon}\left\{ \beta^{2}\xi\left(\frac{l}{N}\right)-I\left(\frac{l}{N}\right)\right\} +o(N)\right)\to0,\text{ as }N\to\infty,
\end{array}
\]
since the supremum is negative for all positive $\varepsilon>0$ when $\beta < \beta_{\xi}$.
Finally the CLT implies that $\frac{\sum_{i=1}^{N}\sigma_{i}}{\sqrt{N}}\to\mathcal{N}\left(0,1\right)$
under $E$, so that taking first the limit $N\to\infty$ and
then the limit $M\to\infty$ the first line the RHS of \eqref{eq: split in 3}, namely
\[
\begin{array}{l}
E\left[\exp\left(\beta^{2}N\xi\left(\frac{\sum_{i=1}^{N}\sigma_{i}}{N}\right)\right)1_{\left\{ \left|\frac{\sum_{i=1}^{N}\sigma_{i}}{\sqrt{N}}\right|\le M\right\} }\right]\\
=E\left[\exp\left(\beta^{2} \alpha_2 \left(\frac{\sum_{i=1}^{N}\sigma_{i}}{\sqrt{N}}\right)^{2}+O\left(\frac{M^3}{\sqrt{N}}\right)\right)1_{\left\{ \left|\frac{\sum_{i=1}^{N}\sigma_{i}}{\sqrt{N}}\right|\le M\right\} }\right],
\end{array}
\]
converges to
\[
\frac{1}{\sqrt{2\pi}}\int e^{\beta^{2}\alpha_2 x^{2}-\frac{x^{2}}{2}}dx=\frac{1}{\sqrt{1-2\alpha_2 \beta^{2}}},
\]
which is then also the limit of $\EN[\hat{Z}_N^2]$.
\end{proof}

The next lemma gives a bound on the tail of the sum appearing in Theorem \ref{thm: sec thm}, and will be used to truncate this sum in the proofs of Theorems \ref{thm: main thm}-\ref{thm: sec thm}.

\begin{lemma}[Bound on tail of cycle count sum]
\label{lem: tail bound} If $\beta < \frac{1}{\sqrt{2\alpha_2}}$ it holds or any $x$ and $K$ that
\begin{equation}\label{eq: tail bound}
\sup_{N\ge1}\PN\left(\left|\sum_{k=K+1}^{\infty}\left\{ C_{N,k}\frac{\mu_{k}}{2k}-\frac{\mu_{k}^{2}}{4k}\right\} \right|\ge x\right)\le \frac{2}{1-2 \alpha_2 \beta^2} \frac{(2\alpha_2\beta^2)^{K+1}}{x^2}.
\end{equation}
\end{lemma}
\begin{proof}Using Lemma \ref{lem: means and variances} we have that
\begin{equation}\label{eq:varcalc}
\begin{split}
    \EN\left[ \left(\sum_{k=K+1}^{\infty}\left\{ C_{N,k}\frac{\mu_{k}}{2k}-\frac{\mu_{k}^{2}}{4k}\right\}\right)^2 \right] &= \EN\left[ \left( \sum_{k=K+1}^{\infty} C_{N,k}\frac{\mu_{k}}{2k} \right)^2\right]+ \left( \sum_{k=K+1}^{\infty} \frac{\mu_{k}^2}{4k} \right)^2\\
    &= \sum_{k=K+1}^{\infty} \frac{\mu_{k}^2}{4k^2}\EN[C_{N,k}^2]+ \left( \sum_{k=K+1}^{\infty} \frac{\mu_{k}^2}{4k}\right)^2\\
&  \le \sum_{k=K+1}^{\infty}\frac{\mu_{k}^2}{2k}+\left( \sum_{k=K+1}^{\infty} \frac{\mu_{k}^2}{4k} \right)^2,
\end{split}
\end{equation}
where the last inequality follows since $\EN[C_{N,k}^{2}]\le2k$ for all $k$ by \eqref{varcnk}. Now observe that $\sum_{k=K+1}^{\infty}\frac{\mu_{k}^2}{2k} \le \frac{1}{\sqrt{1-2 \alpha_2 \beta^2}}( 2\alpha_2 \beta^2)^{K+1} $ by \eqref{eq: muk def}. Then \eqref{eq: tail bound} follows from an application of Chebyshev's inequality.
\end{proof}

We have now prepared all the tools needed to prove Theorem \ref{thm: sec thm}. Before giving the formal proof we give a more detailed heuristic sketch. The
Radon-Nikodym derivative $\hat{Z}_{N}=\frac{d\mathbb{Q}_{N}}{d\mathbb{P}_{N}}$
changes the law of the sequence
\[
C_{N,1},C_{N,2},\ldots,
\]
from approximately independent Gaussian such that $C_{N,k}\sim\mathcal{N}\left(0,2k\right)$
(Proposition \ref{prop: cycle counts under P}) to approximately independent Gaussian such that $C_{N,k}\sim\mathcal{N}\left(\mu_{k},2k\right)$
(Proposition \ref{prop: cycle counts under Q}). 
As mentioned in the introduction, a Radon-Nikodym derivative that changes the law of a sequence
$C_{N,1},C_{N,2},\ldots$ from \emph{exactly} independent
with $C_{N,k}\sim\mathcal{N}\left(0,2k\right)$ to \emph{exactly}
independent with $C_{N,k}\sim\mathcal{N}\left(\mu_{k},2k\right)$
is necessairily equal to
\[
\exp\left(\sum_{k=1}^{\infty}\left\{ C_{N,k}\frac{\mu_{k}}{2k}-\frac{\mu_{k}^{2}}{4k}\right\} \right).
\]

We seek to prove that $\hat{Z}_{N}$ is approximately equal to this
expression, using that the $C_{N,k}$ are approximately Gaussian under
$\mathbb{P}_{N}$ and $\mathbb{Q}_{N}$.

A naive attempt would be to use Chebyshev's inequality 
\[
\PN\left(\left|\hat{Z}_{N}-\exp\left(\sum_{k=1}^{\infty}\left\{ C_{N,k}\frac{\mu_{k}}{2k}-\frac{\mu_{k}^{2}}{4k}\right\} \right)\right|\ge\varepsilon\right)\le\frac{\EN\left[\left|\hat{Z}_{N}-\exp\left(\sum_{k=1}^{\infty}\left\{ C_{N,k}\frac{\mu_{k}}{2k}-\frac{\mu_{k}^{2}}{4k}\right\} \right)\right|^{2}\right]}{\varepsilon^{2}}.
\]
However $\sum_{k=1}^{\infty}\left\{ C_{N,k}\frac{\mu_{k}}{2k}-\frac{\mu_{k}^{2}}{4k}\right\} $
has no well-defined exponential moments (in fact $C_{N,3}$ does not
since a product of three independent Gaussians does not have any finite
exponential moments), and therefore $\exp\left(\sum_{k=1}^{\infty}\left\{ C_{N,k}\frac{\mu_{k}}{2k}-\frac{\mu_{k}^{2}}{4k}\right\} \right)$
does not have a second moment. Thus we instead use an argument involving
subsequential limits to define limiting random variables $\hat{Z}_{\infty},C_{\infty,k}$
where the $C_{\infty,k}$ are exactly Gaussian and independent, so
that $\sum_{k=1}^{\infty}\left\{ C_{\infty,k}\frac{\mu_{k}}{2k}-\frac{\mu_{k}^{2}}{4k}\right\} $
has well-defined exponential moments, and apply the Chebyshev inequality
argument in the limit. 
\begin{proof}[Proof of Theorem \ref{thm: sec thm}]

Let $\varepsilon>0$ be arbitrary. Let $\delta>0$ and assume for
contradiction that
\begin{equation}\label{eq: for contr}
\limsup_{N\to\infty}\mathbb{P}_{N}\left(\left|\log\hat{Z}_{N}-\sum_{k=1}^{\infty}\left\{ C_{N,k}\frac{\mu_{k}}{2k}-\frac{\mu_{k}^{2}}{4k}\right\} \right|\ge\varepsilon\right)\ge\delta.
\end{equation}
To truncate the sum (which must be done to make the above sketch precise) note that because of the inequality
\[
e^{-b}\left|a-b\right|\le\left|e^{a}-e^{b}\right|\text{ for all }a,b\in\mathbb{R},
\]
with $b=\hat{Z}_{N}$ we have for any $K$ and $R$ that
\[
\begin{array}{l}
\mathbb{P}_{N}\left(\left|\log\hat{Z}_{N}-\sum_{k=1}^{\infty}\left\{ C_{N,k}\frac{\mu_{k}}{2k}-\frac{\mu_{k}^{2}}{4k}\right\} \right|\ge\varepsilon\right)\\
\le\mathbb{P}_{N}\left(\left|\sum_{k=K+1}^{\infty}\left\{ C_{N,k}\frac{\mu_{k}}{2k}-\frac{\mu_{k}^{2}}{4k}\right\} \right|\ge\frac{\varepsilon}{2}\right)+\mathbb{P}_{N}\left(\hat{Z}_{N}\ge R\right)\\
\quad\quad+\mathbb{P}_{N}\left(\left|\hat{Z}_{N}-\exp\left(\sum_{k=1}^{K}\left\{ C_{N,k}\frac{\mu_{k}}{2k}-\frac{\mu_{k}^{2}}{4k}\right\} \right)\right|\ge e^{-R}\frac{\varepsilon}{2}\right),
\end{array}
\]

Pick $R$ large enough so that
\[
\mathbb{P}_{N}\left(\hat{Z}_{N}\ge R\right)\le\frac{\mathbb{E}_{N}\left[\hat{Z}_{N}\right]}{R}\overset{\eqref{eq: Z hat mean}}{=}
\frac{1}{R}\le\frac{\delta}{3}.
\]
Also pick $K_{0}=K_{0}\left(\varepsilon,\delta\right)$ large enough
so that if $K\ge K_{0}$ we have
\[
\mathbb{P}_{N}\left(\left|\sum_{k=K+1}^{\infty}\left\{ C_{N,k}\frac{\mu_{k}}{2k}-\frac{\mu_{k}^{2}}{4k}\right\} \right|\ge\frac{\varepsilon}{2}\right)\le\frac{\delta}{3},
\]
by Lemma \ref{lem: tail bound}. Then if (\ref{eq: for contr}) holds we must have for such $K$
that 
\begin{equation}
\limsup_{N\to\infty}\mathbb{P}_{N}\left(\left|\hat{Z}_{N}-\exp\left(\sum_{k=1}^{K}\left\{ C_{N,k}\frac{\mu_{k}}{2k}-\frac{\mu_{k}^{2}}{4k}\right\} \right)\right|\ge e^{-R}\frac{\varepsilon}{2}\right)\ge\frac{\delta}{3}.\label{eq: next for contr}
\end{equation}
If (\ref{eq: next for contr}) holds then there is a subsequence along
which this probability is always at least $\frac{\delta}{6}$. Since
$\hat{Z}_{N}$ is in $L_{1}$ it is tight and there is a further
subsequence $N_{l}$ along which the probability is at least $\frac{\delta}{6}$
and $\hat{Z}_{N_{l}}$ converges in distribution. As $\left( C_{N,1},\ldots , C_{N,K} \right)$ converge in distribution by Proposition \ref{prop: cycle counts under P} it also converges along this subsequence, so that 
\[
\left(\hat{Z}_{N_{l}},C_{N_{l},1},\ldots,C_{N_{l},K}\right)\overset{D}{\to}\left(\hat{Z}_{\infty},C_{\infty,1},\ldots,C_{\infty,K}\right),
\]
where the latter random vector is defined on some auxiliary probability
space with probability $\mathbb{P}_{\infty}$. We must have
\begin{equation}
\mathbb{P}_{\infty}\left(\left|\hat{Z}_{\infty}-\exp\left(\sum_{k=1}^{K}\left\{C_{\infty,k}\frac{\mu_{k}}{2k}-\frac{\mu_{k}^{2}}{4k}\right\}\right)\right|\ge e^{-R}\frac{\varepsilon}{2}\right)\ge\frac{\delta}{6}.\label{eq: next next for contr}
\end{equation}
Note that by Fatou's lemma and Lemma \ref{lem: sec mom}
\begin{equation}\label{eq: Zinf sec mom}
    \mathbb{E}_\infty[\hat{Z}_\infty^2] \le \frac{1}{\sqrt{1-2\alpha_2 \beta^2}}<\infty.
\end{equation}
Because of this this and the fact that the $C_{\infty,k}$ have finite exponential moments since they are exactly
Gaussian under $\mathbb{P}_{\infty}$, we can use the Chebyshev inequality to bound
\begin{equation}\label{eq: chebyshev in limit}
\mathbb{P}_{\infty}\left(\left|\hat{Z}_{\infty}-\exp\left(\sum_{k=1}^{K}\left\{C_{\infty,k}\frac{\mu_{k}}{2k}-\frac{\mu_{k}^{2}}{4k}\right\}\right)\right|\ge e^{-R}\frac{\varepsilon}{2}\right)
\le
\frac{\mathbb{E}_{\infty}\left[\left(\hat{Z}_{\infty}-\exp\left(\sum_{k=1}^{K}C_{\infty,k}\frac{\mu_{k}}{2k}-\frac{\mu_{k}^{2}}{4k}\right)\right)^{2}\right]}{e^{-2R}\varepsilon^{2}/4}.
\end{equation}
We can compute an upper bound of the second moment on the RHS (\ref{eq: chebyshev in limit})
from the convergence in law of $\hat{Z}_{N_{l}},C_{N_{l},k}$ in a way that does not require the existence of a second moment of $C_{N,k}$ as follows. Define
\[
M_{\infty,K}=\mathbb{E}_{\infty}\left[\hat{Z}_{\infty}|\mathcal{F}_{K}\right],
\]
where $\mathcal{F}_{K}$ is the $\sigma$-algebra generated by $C_{\infty,1},\ldots,C_{\infty,K}$.
We now show that in fact 
\[
M_{\infty,K}=\exp\left(\sum_{k=1}^{K}\left\{ C_{\infty,k}\frac{\mu_{k}}{2k}-\frac{\mu_{k}^{2}}{4k}\right\} \right)\text{ a.s.},
\]
which follows if we show that for any continuous bounded $f$
\[
\mathbb{E}_{\infty}\left[f\left(C_{\infty,1},\ldots,C_{\infty,K}\right)\hat{Z}_{\infty}\right]=\mathbb{E}_{\infty}\left[f\left(C_{\infty,1},\ldots,C_{\infty,K}\right)M_{\infty,K}\right].
\]
This can be proven from the convergence in law of $C_{N,k}$ (without
requiring that $C_{N,k}$ has a  second moment) because by definition 
$\hat{Z}_{N}=\frac{d\mathbb{Q}_{N}}{d\mathbb{P}_{N}}$ implying
\[
\lim_{l\to\infty}\mathbb{E}_{N_{l}}\left[f\left(C_{N_{l},1},\ldots,C_{N_{l},K}\right)\hat{Z}_{N_{l}}\right]=\lim_{l\to\infty}\mathbb{Q}_{N_{l}}\left[f\left(C_{N_{l},1},\ldots,C_{N_{l},K}\right)\right],
\]
and then by Proposition \ref{prop: cycle counts under Q} the right-hand side equals
\[
\mathbb{E}_{\infty}\left[f\left(C_{\infty,1}+\mu_{1},\ldots,C_{\infty,K}+\mu_{K}\right)\right],
\]
which in turns equals $\mathbb{E}_{\infty}\left[f\left(C_{\infty,1},\ldots,C_{\infty,K}\right)M_{\infty,K}\right]$
since $M_{\infty,K}$ is precisely the Radon-Nikodym derivative which
changes the mean of $C_{\infty,k}$ to $\mu_{k}$ while leaving all
covariances fixed.

Thus the second moment on the RHS of \eqref{eq: chebyshev in limit} is
\[
\mathbb{E}_{\infty}\left[\left(\hat{Z}_{\infty}-M_{\infty,K}\right)^{2}\right]
=
 \mathbb{E}_{\infty}\left[\hat{Z}_{\infty}^{2}\right]-\mathbb{E}_{\infty}\left[M_{\infty,K}^{2}\right],
\]
where the equality follows since $M_{\infty,K}=\mathbb{E}_{\infty}\left[\hat{Z}_{\infty}|\mathcal{F}_{K}\right]$.
By explicit computation 
\[
\mathbb{E}_{\infty}\left[M_{\infty,K}^{2}\right]=\exp\left(  \sum_{k=1}^{K} \frac{\left(2 \alpha_2 \beta^2 \right)^k}{2k}\right)=\frac{1}{\sqrt{1-2 \alpha_2 \beta^2}}+o_{K}\left(1\right),
\]
which together with \eqref{eq: Zinf sec mom} gives
$$\mathbb{E}_{\infty}\left[\hat{Z}_{\infty}^{2}\right]-\mathbb{E}_{\infty}\left[M_{\infty,K}^{2}\right] = o_K(1).$$
Thus we obtain from (\ref{eq: chebyshev in limit}) that for
any $K\ge K_{0}$
\[
\mathbb{P}_{\infty}\left(\left|\hat{Z}_{\infty}-\exp\left(\sum_{k=1}^{K}C_{\infty,k}\frac{\mu_{k}}{2k}-\frac{\mu_{k}^{2}}{4k}\right)\right|\ge\ \frac{e^{-R}\varepsilon}{2}\right)\le\frac{o_{K}\left(1\right)}{e^{-2R}\varepsilon^{2}}.
\]
We may now for any $R$ and $\varepsilon>0$ pick a $K\ge K_{0}$
large enough so that this contradicts (\ref{eq: next next for contr}).
Thus (\ref{eq: for contr}) can not hold for any $\delta>0$. Thus
in fact for all $\varepsilon>0$
\[
\mathbb{P}_{N}\left(\left|\log\hat{Z}_{N}-\sum_{k=1}^{\infty}\left\{ C_{N,k}\frac{\mu_{k}}{2k}-\frac{\mu_{k}^{2}}{4k}\right\} \right|\ge\varepsilon\right)\to0,
\]
which together with \eqref{eq: Z hat def} and \eqref{eq: ann Z} implies \eqref{eq: FE estimate}.
\end{proof}

\section{Derivation of Theorem \ref{thm: main thm} from Theorem \ref{thm: sec thm}}

It only remains to derive Theorem \ref{thm: main thm} from Theorem \ref{thm: sec thm} and the asymptotic normality of the $C_{N,k}$.

\begin{proof}[Proof of Theorem \ref{thm: main thm}]
Let $
s^2_{K}=\sum_{k=1}^{K}\frac{\mu_{k}^{2}}{2k}$,
and note that
\[
s^2_{K}\to 
\sum_{k=1}^{\infty}\frac{\mu_{k}^{2}}{2k}
=\sum_{k=1}^{\infty}\frac{\left(2\alpha\beta_{2}^{2}\right)^{k}}{2k}=-\frac{1}{2}\log\left(1-2\alpha_{2}\beta^{2}\right)=s^2.
\]
Theorem \ref{thm: main thm} follows from from Theorem \ref{thm: sec thm} once we have shown that
\begin{equation}\label{eq: thm1_1 to show}
\sum_{k=1}^{\infty}\frac{2\mu_{k}C_{N,k}-\mu_{k}^{2}}{4k}\overset{D}{\to}\mathcal{N}\left(-\frac{1}{2}s^2,s^2\right)\text{ as }N\to\infty \text{ under } \PN.
\end{equation}
Proposition \ref{prop: cycle counts under P} implies that for any $K$
\begin{equation}\label{eq: trunc sum dist conv}
\sum_{k=1}^{K}\frac{2\mu_{k}C_{N,k}-\mu_{k}^{2}}{4k}\to\mathcal{N}\left(-\frac{1}{2}s^2_K,s^2_K\right).
\end{equation}
Now for any $z \in \mathbb{R}$ 
\[
\begin{split}
&\PN\left(\sum_{k=1}^{\infty}\frac{2\mu_{k}C_{N,k}-\mu_{k}^{2}}{4k}\le z\right)\\
&\le\PN\left(\sum_{k=1}^{K}\frac{2\mu_{k}C_{N,k}-\mu_{k}^{2}}{4k}\le z+\varepsilon\right)+\PN\left(\left|\sum_{k=K+1}^{\infty}\frac{2\mu_{k}C_{N,k}-\mu_{k}^{2}}{4k}\right|\ge\varepsilon\right).
\end{split}
\]
Taking first the limit $N\to\infty$ and using \eqref{eq: trunc sum dist conv} and Lemma \ref{lem: tail bound}, and then
the limit $K\to\infty$ on the right hand side we obtain
\[
\limsup_{N\to\infty}\mathbb{P}\left(\sum_{k=1}^{\infty}\frac{2\mu_{k}C_{N,k}-\mu_{k}^{2}}{4k}\le z\right)\le F\left(z\right),
\]
where $F$ is the CDF of $\mathcal{N}\left(-s^2/2,s^2\right)$.
We get the corresponding lower bound similarly, proving \eqref{eq: thm1_1 to show} and therefore Theorem \ref{thm: main thm}.
\end{proof}

\printbibliography[
heading=bibintoc,
title={References}
] 

\end{document}